\documentclass[11pt,reqno,a4paper]{amsart}
\usepackage[british]{babel}

\usepackage{amsmath}
\usepackage{amsfonts}
\usepackage{amssymb}
\usepackage{pdfsync}
\usepackage{bbm}
\usepackage{dsfont}
\usepackage{a4wide}

\theoremstyle{plain}
\newtheorem{theorem}{Theorem}[section]

\newtheorem{cor}[theorem]{Corollary}
\newtheorem*{cor*}{Corollary}

\newtheorem{lemma}[theorem]{Lemma}

\newtheorem{prop}[theorem]{Proposition}

\newtheorem*{theoremA}{Theorem A}
\newtheorem*{theoremB}{Theorem B}
\newtheorem*{theoremC}{Theorem C}
\newtheorem*{corollaryD}{Corollary D}

\newtheorem*{remark*}{Remark}

\theoremstyle{definition}
\newtheorem{definition}[theorem]{Definition}
\newtheorem*{definition*}{Definition}

\numberwithin{equation}{section}

\newcommand{\R}{\mathbb{R}}
\newcommand{\N}{\mathbb{N}}

\newcommand{\ii}{\mathbf{i}}
\newcommand{\jj}{\mathbf{j}}

\newcommand{\A}{\mathbf{A}}
\newcommand{\e}{\mathbf{e}}

\DeclareMathOperator{\graph}{graph}
\DeclareMathOperator{\dist}{dist}
\DeclareMathOperator{\Id}{Id}
\DeclareMathOperator{\interior}{int}

\begin{document}
\title{On the dimension of the graph of the classical Weierstrass function}

\author{Krzysztof Bara\'nski}
\address{Krzysztof Bara\'nski, Institute of Mathematics, University of Warsaw, ul. Banacha 2, 02-097 Warszawa, Poland} \email{baranski@mimuw.edu.pl}

\author{Bal\'azs B\'ar\'any}
\address{Bal\'azs B\'ar\'any, Institute of Mathematics, Polish Academy of Sciences, ul. \'Sniadeckich 8, 00-956 Warszawa, Poland \&
Budapest University of Technology and Economics, MTA-BME Stochastics Research Group, P.O.Box 91, 1521 Budapest, Hungary} \email{balubsheep@gmail.com}

\author{Julia Romanowska}
\address{Julia Romanowska, Institute of Mathematics, University of Warsaw, ul. Banacha 2, 02-097 Warszawa, Poland} \email{romanoju@mimuw.edu.pl}

\thanks{Krzysztof Bara\'nski was partially supported by
Polish NCN Grant N~N201~607940. Bal\'azs B\'ar\'any was supported by the Warsaw Center of Mathematics and Computer Science -- KNOW}

\begin{abstract}
This paper examines dimension of the graph of the famous Weierstrass non-differentiable function
\[
W_{\lambda, b} (x)  = \sum_{n=0}^{\infty}\lambda^n\cos(2\pi b^n x)
\]
for an integer $b \ge 2$ and $1/b < \lambda < 1$. We prove that for every $b$ there exists (explicitly given) $\lambda_b \in (1/b, 1)$ such that the Hausdorff dimension of the graph of $W_{\lambda, b}$ is equal to $D = 2+\frac{\log\lambda}{\log b}$ for every $\lambda\in(\lambda_b,1)$. We also show that the dimension is equal to $D$ for almost every $\lambda$ on some larger interval. This partially solves a well-known thirty-year-old conjecture. Furthermore, we prove that the Hausdorff dimension of the graph of the function
\[
f (x)  = \sum_{n=0}^{\infty}\lambda^n\phi(b^n x)
\]
for an integer $b \ge 2$ and $1/b < \lambda < 1$ is equal to $D$ for a typical $\mathbb Z$-periodic $C^3$ function $\phi$.
\end{abstract}


\maketitle

\section{Introduction and statements}\label{sec:intro}
This paper is devoted to the study of dimension of the graphs of functions of the form
\begin{equation}\label{f^phi}
f_{\lambda,b}^{\phi}(x)=\sum_{n=0}^{\infty}\lambda^n\phi (b^n x)
\end{equation}
for $x\in\R$, where $b>1$, $1/b <\lambda<1$ and $\phi: \mathbb \R \to \mathbb \R$ is a non-constant $\mathbb Z$-periodic Lipschitz continuous piecewise $C^1$ function. The famous example
\[
W_{\lambda,b}(x)=\sum_{n=0}^{\infty}\lambda^n\cos (2\pi b^n x),
\]
for $\phi(x)=\cos(2\pi x)$, was introduced by Weierstrass in 1872 as one of the first examples of a continuous nowhere differentiable function on the real line. In fact, Weierstrass proved the non-differentiability for some values of the parameters, while the complete proof was given by Hardy \cite{Ha} in 1916. Later, starting from the work of Besicovitch and Ursell \cite{BU}, the graphs of $f_{\lambda,b}^{\phi}$ and related functions were studied from a geometric point of view as fractal curves in the plane.

The graph of a function $f_{\lambda, b}^\phi$ of the form \eqref{f^phi} is approximately self-affine with scales $\lambda$ and $1/b$, which suggests that its dimension should be equal to
\[
D = 2 + \frac{\log \lambda}{\log b}.
\]
Indeed, it follows from \cite{HL,KMPY} that every function of the form \eqref{f^phi} either is piecewise $C^1$ (and hence the dimension of its graph is $1$), or the box dimension of its graph is equal to $D$. More generally, the graph of every function of the form
\begin{equation}\label{f}
f(x) = \sum_{n=0}^{\infty} \lambda^n\phi(b^n x + \theta_n),
\end{equation}
where $\theta_n \in \R$, has upper box dimension at most $D$. In \cite{He} it was proved  that the box dimension of the graph of a function of the form \eqref{f}
is equal to $D$ in the case when $b$ is a transcendental number. 

However, the question of determining the Hausdorff dimension of the graph turned out to be much more complicated.
In 1986, Mauldin and Williams \cite{MW} proved that if a function $f$ has the form \eqref{f}, then for given $D$ there exists a constant $C > 0$ such that the Hausdorff dimension of the graph is larger than $D - C/\log b$
for large $b$. Shortly after, Przytycki and Urba\'nski showed in \cite{PU} that
the Hausdorff dimension of the graph of $f$ is larger than $1$ under some weaker assumptions.

Let us note that if $b \ge 2$ is an integer, then the graph of the function $f^\phi_{\lambda, b}$ from \eqref{f^phi} is an invariant repeller for the expanding dynamical system
$\Phi: \R/\mathbb Z \times \R \to \R/\mathbb Z \times \R$,
\begin{equation}\label{Phi}
\Phi(x, y) = \left(bx \text{ (mod } 1),\; \frac{y- \phi(x)}{\lambda}\right)
\end{equation}
with Lyapunov exponents $0 < -\log \lambda < \log b$, which allows to use the methods of ergodic theory of smooth dynamical systems.

On the other hand, if we replace $b^n$ by a sequence $b_n$ with $b_{n+1}/b_n \to \infty$, then the question of determining the Hausdorff and box dimension of graphs of functions \eqref{f} can be solved completely, as proved recently by Carvalho \cite{Ca} and Bara\'nski \cite{B12}. In this case the Hausdorff and upper box dimension need not coincide.

In 1992, Ledrappier \cite{Led} proved that for $\phi(x) = \dist(x, \mathbb Z)$ and $b = 2$, the Hausdorff dimension of the graph of the function $f^\phi_{\lambda, b}$ from \eqref{f^phi} is equal to $D$ provided the infinite Bernoulli convolution $\sum_{n=0}^\infty \pm 2^{(1-D)n}$, with $\pm$ chosen independently with probabilities $(1/2, 1/2)$, has absolutely continuous distribution (by the result of Solomyak \cite{So2}, this holds for almost all $D \in (1, 2)$ with respect to the Lebesgue measure). Analogous result was showed by Solomyak in \cite{So} for a class of functions $\phi$ with discontinuous derivative.

However, the question of determining the Hausdorff dimension of the classical Weierstrass function $W_{\lambda, b}$, for $\phi(x)=\cos(2\pi x)$ and integer $b \ge 2$, has remained open. The conjecture that it is equal to $D$ was formulated by Mandelbrot in 1977 \cite{Ma} and then repeated in many papers, see e.g.~\cite{BL,F1,Hu,Led,MW, PU} and the references therein.
Przytycki and Urba\'nski showed in \cite{PU} that it is greater than $1$.
Rezakhanlou \cite{Rez} proved that the packing dimension of the graph of $W_{\lambda, b}$ is equal to $D$ and in \cite{HL}, Hu and Lau showed the same for the so-called $K$-dimension (both are not smaller than the Hausdorff dimension).

In 1998, Hunt \cite{Hu} proved that if one considers the numbers $\theta_n$ in \eqref{f} as independent random variables with uniform distribution on $[0,1]$, then for many functions $\phi$, including $\phi(x) =\cos(2\pi x)$, the Hausdorff dimension of the graph of $f$ is equal to $D$ almost surely (analogous result for the box dimension was proved by Heurteaux in \cite{He}).

It is interesting to notice that for an integer $b \ge 2$, the function $W_{\lambda, b}$ is the real part of the lacunar (Hadamard gaps) power series
\[
w(z) = \sum_{n=0}^{\infty}\lambda^n z^{b^n}, \qquad z \in \mathbb C, \; |z| \leq 1
\]
on the unit circle, which relates the problem to harmonic analysis and boundary behaviour of analytic maps. For instance, it was proved by Salem and Zygmund \cite{SZ} and Kahane, Weiss and Weiss \cite{KWW} that for $\lambda$ sufficiently close to $1$, the image of the unit circle under $w$ is a Peano curve, i.e.~it covers an open subset of the plane. Moreover, Belov \cite{Be} and Bara\'nski \cite{Ba} showed that in this case the map $w$ does not preserve (forwardly) Borel sets on the unit circle. The complicated topological boundary behaviour of $w$ was also studied recently by Dong, Lau and Liu in \cite{DLL}.

In this paper we solve the Mandelbrot conjecture for every integer $b \ge 2$ in the case when $\lambda$ is sufficiently close to $1$. More precisely, we prove that the Hausdorff dimension of the graph of the classical Weierstrass function $W_{\lambda,b}$ is equal to $D$ for every integer $b \ge 2$, provided $\lambda \in (\lambda_b, 1)$ for some (explicitly given) $\lambda_b \in (1/b, 1)$, with $\lambda_b$ tending to $1/\pi$ for $b \to \infty$. We extend the result for almost every $\lambda$ on some larger interval $(\tilde\lambda_b, 1)$ with $\tilde\lambda_b \sqrt{b} \to 1/\sqrt{\pi}$ as $b \to \infty$.

Moreover, we prove that the Hausdorff dimension of the graph of the function $f^\phi_{\lambda,b}$ from \eqref{f^phi} is equal to $D$ for a typical $\mathbb Z$-periodic $C^3$ function $\phi$. In particular, for given integer $b \ge 2$ and $\lambda \in (1/b, 1)$ this holds for an open and dense subset of functions $\phi$ in $C^3$ topology.

In fact, in all mentioned results we show that the measure $\mu_{\lambda,b}^{\phi}$ has local dimension $D$, where
\begin{equation}\label{def_mu}
\mu_{\lambda,b}^{\phi} = (\Id, f_{\lambda, b}^{\phi})_*\mathcal{L}|_{[0,1]}
\end{equation}
is the lift of the Lebesgue measure $\mathcal L$ on $[0,1]$ to the graph of $f^\phi_{\lambda,b}$.

\begin{definition} We say that a Borel measure $\mu$ in a metric space $X$ has local dimension $d$ at a point $x \in X$, if
\[
\lim_{r \to 0} \frac{\log\mu(B_r(x))}{\log r} = d,
\]
where $B_r(x)$ denotes the ball of radius $r$ centered at $x$. If the local dimension of $\mu$ exists and is equal to $d$ at $\mu$-almost every $x$, then we say that $\mu$ has local dimension $d$ and write $\dim \mu = d$.

If $\dim \mu = d$, then every set of positive measure $\mu$ has Hausdorff dimension at least $d$.

\end{definition}

Let $\dim_H$ and $\dim_B$ denote, respectively, the Hausdorff and box dimension (for the definition and basic properties of the Hausdorff and box dimension we refer to \cite{F1,Mat}). As mentioned above, it is well-known that $\dim_B \graph f_{\lambda,b}^{\phi} \le D$. Since $\dim_H \leq \dim_B$, to show that the Hausdorff dimension of $\graph f_{\lambda,b}^{\phi}$ is equal to $D$, it is sufficient to prove $\dim \mu_{\lambda,b}^{\phi} = D$.

\begin{definition} For $r = 3, 4, \ldots, \infty$
consider the space of $\mathbb{Z}$-periodic $C^r$ functions $\phi: \mathbb R \to \mathbb R$, treated as functions on $\mathbb S^1 = \mathbb R/\mathbb Z$. For given integer $b \geq 2$ let $\mathcal F_b \subset (1/b, 1) \times C^3(\mathbb S^1, \mathbb R)$ be the set of pairs $(\lambda, \phi)$, for which $\dim \mu_{\lambda,b}^{\phi}= D$. We denote by $\interior\mathcal F_b$ the interior of $\mathcal F_b$ with respect to the product of the Euclidean topology on $(1/b, 1)$ and $C^3$ topology on $C^3(\mathbb S^1, \mathbb R)$.
\end{definition}

In particular, for every $(\lambda, \phi)\in\mathcal{F}_b$,
\[
\dim_H \graph f_{\lambda,b}^{\phi}=\dim_B \graph f_{\lambda,b}^{\phi}=D.
\]

The first result of the paper is the following.

\begin{theoremA} For every integer $b \ge 2$ and $\lambda \in (1/b, 1)$ there exist functions $\phi_1, \ldots, \phi_m \in C^\infty(\mathbb S^1, \mathbb R)$, for some positive integer $m$, such that for every function $\phi \in C^3(\mathbb S^1, \mathbb R)$,
\[
\left(\lambda, \phi + \sum_{j=1}^m t_j \phi_j\right) \in \interior\mathcal F_b
\]
for Lebesgue almost every $(t_1, \ldots, t_m) \in \mathbb R^m$.
\end{theoremA}

This easily implies the following corollary.

\begin{cor*}
For every integer $b \ge 2$ the set $\mathcal F_b$ contains an open and dense subset of $(1/b, 1) \times C^3(\mathbb S^1, \mathbb R)$. Moreover, for every  $\lambda \in (1/b, 1)$ and $r =3, 4, \ldots, \infty$, the set
\[
\{\phi \in C^r(\mathbb S^1, \mathbb R): (\lambda, \phi) \in \interior\mathcal F_b\}
\]
is an open and dense subset of $C^r(\mathbb S^1, \mathbb R)$.
\end{cor*}

Now we state the results about the graph of the classical Weierstrass function $W_{\lambda, b}$, for an integer $b \ge 2$, $\lambda \in (1/b, 1)$ and $\phi(x)=\cos(2\pi x)$. For simplicity, we will write $\mu_{\lambda,b}=\mu_{\lambda,b}^{\phi}$ in this case.

\begin{theoremB} For every integer $b \ge 2$,
\[
\dim \mu_{\lambda,b}= 2+\frac{\log\lambda}{\log b}
\]
for every $\lambda \in (\lambda_b, 1)$, where $\lambda_b$ is equal to the unique zero of the function
\[
h_b(\lambda) =
\begin{cases}
\frac{1}{4\lambda^2(2\lambda-1)^2} +
\frac{1}{16\lambda^2(4\lambda-1)^2} - \frac{1}{8\lambda^2} + \frac{\sqrt{2}}{2\lambda} - 1 & \text{for } b = 2\\
\frac{1}{(b\lambda - 1)^2} + \frac{1}{(b^2\lambda - 1)^2} - \sin^2\left(\frac{\pi}{b}\right)  & \text{for } b \geq 3
\end{cases}
\]
on the interval $(1/b, 1)$. In particular,
\[
\dim_H \graph W_{\lambda,b}=\dim_B \graph W_{\lambda,b}=D
\]
for every $\lambda \in (\lambda_b, 1)$.
\end{theoremB}

Using the Pollicott--Simon--Peres--Solomyak transversality methods, we can extend the result for almost every $\lambda$ on a larger interval. To state the next theorem, we need to recall some definitions related to so-called \mbox{$(*)$-}functions considered in the study of infinite Bernoulli convolutions (see e.g.~\cite{peso, peso2, So}). For $\beta \geq 1$ let
$$\mathcal{G}_\beta=\left\{ g(t)=1+\sum_{n=1}^{\infty}g_nt^n, \: g_n\in[-\beta,\beta] \text{ for } n \geq 1\right\}.$$
Let $y(\beta)$ be the smallest possible value of positive double roots of functions in $\mathcal{G}_\beta$, i.e.
\[
y(\beta)=\inf\left\{t>0: \text{there exists }g \in\mathcal{G}_\beta\text{ such that }g(t)=g'(t)=0\right\}.
\]

\begin{theoremC}
For every integer $b \ge 2$,
\[
\dim \mu_{\lambda,b}= D
\]
for Lebesgue almost every $\lambda \in (\tilde\lambda_b, 1)$, where $\tilde \lambda_b$ is equal to the unique root of the equation
\[
y\left(\frac{1}{\sqrt{\sin^2(\pi/b) - 1/(b^2\lambda - 1)^2}}\right) = \frac{1}{b\lambda}
\]
on the interval $(1/b, 1)$. In particular,
\[
\dim_H \graph W_{\lambda,b}=\dim_B \graph W_{\lambda,b}=D
\]
for Lebesgue almost every $\lambda \in (\tilde\lambda_b, 1)$.
\end{theoremC}

Estimating the numbers $\lambda_b$ and $\tilde\lambda_b$ in the above theorems, we obtain the following.
\begin{corollaryD}
\begin{align*}
\dim_H \graph W_{\lambda,2} &= 2+\frac{\log\lambda}{\log 2} & &\text{for every } \lambda \in (0.9352, 1) \text{ and almost every } \lambda \in (0.81, 1),\\
\dim_H \graph W_{\lambda,3} &= 2+\frac{\log\lambda}{\log 3} & &\text{for every } \lambda \in (0.7269, 1) \text{ and almost every } \lambda \in (0.55, 1),\\
\dim_H \graph W_{\lambda,4} &= 2+\frac{\log\lambda}{\log 4} & &\text{for every } \lambda \in (0.6083, 1) \text{ and almost every } \lambda \in (0.44, 1).
\end{align*}
For every $b \geq 5$,
\begin{align*}
\dim_H \graph W_{\lambda,b} &= 2+\frac{\log\lambda}{\log b} & \text{ for every } \lambda \in (0.5448, 1) \text{ and almost every } \lambda \in (1.04/\sqrt{b}, 1).
\end{align*}
\end{corollaryD}

Obviously, using Theorems~B and~C, one can get better estimates for $b \geq 5$. In fact,
\[
\lambda_b \to \frac 1 \pi \quad \text{and} \quad \tilde\lambda_b\sqrt{b} \to \frac{1}{\sqrt{\pi}} \qquad \text{as } b \to \infty,
\]
see Lemmas~\ref{h} and~\ref{tilde}.

In the proofs we use the techniques of ergodic theory of non-uniformly hyperbolic smooth dynamical systems on manifolds (Pesin theory) developed by Ledrappier and Young in \cite{LY} and applied by Ledrappier in \cite{Led} to study the graphs of the Weierstrass-type functions. Theorems~A and~B rely on the results proved by Tsujii in \cite{T} about the SBR measure for some smooth Anosov endomorphisms on the cylinder $\mathbb S^1 \times \mathbb R$. The proof of Theorem~C uses the Peres--Solomyak transversality techniques developed in \cite{peso,peso2}.

\section{Ledrappier--Young theory and graphs of functions} \label{sec:back}
We consider $\graph f_{\lambda,b}^{\phi}$ as an invariant repeller of the dynamical system \eqref{Phi}. The Ledra\-ppier--Young theory in \cite{LY} is valid for smooth diffeomorphisms, so to apply it for $\Phi$ one considers the inverse limit (alternatively, it is possible to use analogous theory for smooth endomorphisms developed by Qian, Xie and Zhu in \cite{Qian}).

For the reader's convenience, let us recall the Ledrappier--Young results from \cite{Led, LY} applied for the graph of $f_{\lambda,b}^{\phi}$. (Note that the quoted results are formulated in \cite{Led} for $b=2$. However, the theory is valid for any integer $b \ge 2$.)

Take an integer $b \ge 2$, $\lambda \in (1/b, 1)$ and a non-constant continuous $\mathbb Z$-periodic piecewise $C^1$ function $\phi: \mathbb R \to \mathbb R$. Consider the symbolic space
\[
\Sigma=\left\{0,\ldots, b-1\right\}^{\mathbb Z^+}
\]
and the Bernoulli measure $\mathbb{P}$ on $\Sigma$ with uniform probabilities $\left\{\frac{1}{b},\ldots, \frac{1}{b}\right\}^{\mathbb Z^+}$. Define the inverse of the map $\Phi$ from \eqref{Phi} as the map $F: [0,1] \times \R \times \Sigma \to [0,1] \times \R \times \Sigma$,
\[
F(x,y, \ii) = \left(\frac{x}{b} + \frac{i_1}{b},\; \lambda y +\phi\left(\frac{x}{b} + \frac{i_1}{b}\right), \;\sigma(\ii)\right),
\]
where $\ii = (i_1, i_2, \ldots)$ and $\sigma$ is the left-side shift on $\Sigma$. We have
\[
F(\graph f_{\lambda,b}^{\phi} \times \Sigma) = \graph f_{\lambda,b}^{\phi} \times \Sigma, \qquad F_*(\mu_{\lambda,b}^{\phi} \times \mathbb P) = \mu_{\lambda,b}^{\phi} \times \mathbb P.
\]
Defining
\[
F_i(x,y) = \left(\frac{x}{b} + \frac{i}{b},\; \lambda y +\phi\left(\frac{x}{b} + \frac{i}{b}\right)\right)
\]
for $i \in \{0, \ldots, b-1\}$, we have
\[
DF_i (x,y) =
\begin{bmatrix}
1/b & 0\\
\phi'(x/b + i/b)/b & \lambda
\end{bmatrix}.
\]
Consider products of these matrices, which arise by composing the maps $F_{i_1}, F_{i_2}, \ldots$ for given $\ii = (i_1, i_2, \ldots)$. By a direct check, the Lyapunov exponents of the system are equal to $-\log b < \log \lambda < 0$, the foliation of $(0, 1) \times \mathbb R$ into vertical lines is invariant and $F_i$ contract its fibers affinely with exponent $\log \lambda$, and there is exactly one strong stable direction in $\R^2$ (corresponding to the exponent $-\log b$), given by
\[
\mathcal J_{x,\ii} =
\begin{bmatrix}
1\\
Y_{x,\gamma}(\ii)
\end{bmatrix},
\]
where
\begin{equation}\label{gamma}
\gamma=\frac{1}{b\lambda}
\end{equation}
and
\begin{equation}
\label{eproj}
Y_{x,\gamma}(\ii)=-\sum_{n=1}^{\infty}\gamma^n\phi'\left(\frac{x}{b^n}+\frac{i_1}{b^n}+\cdots+\frac{i_n}{b}\right)
\end{equation}
for $\ii = (i_1, i_2, \ldots)$. (The formula \eqref{gamma}, relating $\gamma$ to $\lambda$, will be used throughout the paper.) In fact,
\[
D F_{i_1}(x,y)(\mathcal J_{x,\ii}) = \frac{1}{b}\mathcal J_{x/b + i_1/b,\:\sigma(\ii)}.
\]
Note that $\mathcal J_{x,\ii}$ does not depend on $y$. For given $\ii \in \Sigma$, the vector field $\mathcal J_{x,\ii}$ defines locally a $C^{1+\varepsilon}$ foliation of $(0,1) \times \R$ into strong stable manifolds $\Gamma_{x,y,\ii}= \Gamma_{y,\mathbf{i}}(x)$, which are parallel $C^{1+\varepsilon}$ curves (graphs of functions of the first coordinate).

For the measure $\mu = \mu_{\lambda, b}^{\phi}$ defined in \eqref{def_mu}, there exists a system of conditional measures $\mu_{x,y,\ii}$ on $\Gamma_{x,y,\ii}$, associated to the foliation $\{\Gamma_{x,y,\ii}\}$ treated as a measurable partition. Take a vertical line $\ell$ and let $\nu_{x,\ii}$ (called transversal measure) be the projection of $\mu_{\lambda, b}$ to $\ell$ along the curves $\Gamma_{x,y,\ii}$, $y \in \R$. The following result is a part of the Ledrappier--Young theory from \cite{LY} (see also \cite[Proposition~2]{Led}).

\begin{theorem}[Ledrappier--Young]
The local dimension of the measure $\mu$ exists and is constant $\mu$-almost everywhere. The local dimension of the measure $\mu_{x,y,\ii}$ exists, is constant $\mu_{x,y,\ii}$-almost everywhere, and is constant for $(\mu \times \mathbb P)$-almost every $(x,y,\ii)$. The local dimension of the measure $\nu_{x,\ii}$ exists, is constant $\nu_{x,\ii}$-almost everywhere, and is constant for $(\mathcal L \times \mathbb P)$-almost every $(x, \ii)$, where $\mathcal L$ is the Lebesgue measure. Moreover,
\[
\dim \mu = \dim\mu_{x,y,\ii} + \dim\nu_{x,\ii}
\]
and
\[
\log b\dim\mu_{x,y,\ii}  - \log \lambda\dim\nu_{x,\ii}  = \log b.
\]
\end{theorem}
The latter is a ``conditional'' version of the Pesin entropy formula. As a corollary, one gets
\begin{equation}\label{dim}
\dim \mu = 1 + \left(1 + \frac{\log \lambda}{\log b}\right) \dim\nu_{x,\ii} = 1 + (D-1)\dim\nu_{x,\ii}.
\end{equation}

In \cite{Led}, Ledrappier proved a kind of the Marstrand-type projection theorem, showing that if the distribution of angles of directions $\mathcal J_{x,\ii}$ has dimension $1$, then the dimension of the transversal measure is also equal to $1$. More precisely, for the measure
\begin{equation}\label{def_m}
m_{x,\gamma}=\left(Y_{x,\gamma}\right)_*\mathbb{P},
\end{equation}
he proved the following.

\begin{theorem}[Ledrappier, \cite{Led}]\label{pled}
Let $\gamma\in(1/b,1)$. If $\dim m_{x,\gamma} = 1$ for Lebesgue almost every $x\in(0,1)$, then $\dim\nu_{x,\ii}=1$.
\end{theorem}

In view of \eqref{dim}, this implies the following corollary.

\begin{cor}\label{cor_led}
Let $b \ge 2$ be an integer, $\lambda \in (1/b, 1)$ and let $\phi: \mathbb R \to \mathbb R$ be a non-constant continuous $\mathbb Z$-periodic piecewise $C^1$ function.
If $\dim m_{x,\gamma} = 1$ for Lebesgue almost every $x\in(0,1)$, where $\gamma = 1/(b\lambda)$, then
\[
\dim \mu_{\lambda, b}^{\phi} = D.
\]
\end{cor}
In fact, we will show that under the assumptions of Theorems~A--C, the measure $m_{x,\gamma}$ is absolutely continuous with respect to the Lebesgue measure for Lebesgue almost every $x\in(0,1)$, which is a stronger property.

\section{Tsujii results and proof of Theorem~{\rm A}} \label{sec:thmA}

In the proofs of Theorems~A and~B we use results due to Tsujii \cite{T}.
He considered smooth Anosov skew products $T: \mathbb S^1 \times \mathbb R \to \mathbb S^1 \times \mathbb R$ of the form
\[
T(x,y) = (bx, \gamma y + \psi(x))
\]
for an integer $b \ge 2$, a real number $\gamma \in (1/b, 1)$ and a $C^2$ function $\psi$ on $\mathbb S^1 =\mathbb R/\mathbb Z$. For the map $T$ there exists an ergodic SBR measure, i.e.~a measure $\vartheta$ on $\mathbb S^1 \times \mathbb R$ such that for every continuous function $g:\mathbb S^1 \times \mathbb R \to \mathbb R$,
\[
\lim_{n\to \infty} \frac{1}{n} \sum_{i = 0}^{n-1} g(T^i(x,y)) = \int g\:  d\vartheta
\]
for Lebesgue almost every $(x,y) \in \mathbb S^1 \times \mathbb R$. The SBR measure $\vartheta$ has a straightforward description -- one can check that
\begin{equation}\label{nu}
\vartheta = \Psi_*(\mathcal L|_{\mathbb S^1} \times \mathbb P),
\end{equation}
where $\Psi: \mathbb S^1 \times \Sigma \to \mathbb S^1 \times \mathbb R$,
\[
\Psi(x, \ii) = (x, S(x,\ii)),
\]
for
\[
S(x, \ii) = \sum_{n=1}^{\infty}\gamma^{n-1}\psi\left(\frac{x}{b^n}+\frac{i_1}{b^n}+\cdots+\frac{i_n}{b}\right)
\]
and $\mathcal L$ is the Lebesgue measure (for details, see \cite{T}).

\begin{definition} For given integer $b\geq2$, let $\mathcal{D} \subset (1/b, 1) \times C^2(\mathbb S^1, \mathbb R)$ be the set of pairs $(\gamma, \psi)$, for which the SBR measure $\vartheta$ for the map $T$ is absolutely continuous with respect to the Lebesgue measure on $\mathbb S^1 \times \mathbb R$, and let $\mathcal D^\circ$ be the interior of $\mathcal D$ with respect to the product of the Euclidean and $C^2$ topology.
\end{definition}

In \cite{T}, Tsujii proved the following result.

\begin{theorem}[Tsujii, \mbox{\cite[Theorem 1]{T}}]\label{tsujiicor}
For every integer $b \ge 2$ and $\gamma \in (1/b, 1)$ there exist functions $\psi_1, \ldots, \psi_m \in C^\infty(\mathbb S^1, \mathbb R)$, for some positive integer $m$, such that for every function $\psi \in C^2(\mathbb S^1, \mathbb R)$,
\[
\left(\gamma, \psi + \sum_{j=1}^m t_j \psi_j\right) \in \mathcal D^\circ
\]
for Lebesgue almost every $(t_1, \ldots, t_m) \in \mathbb R^m$.
\end{theorem}

In the proof of this theorem, Tsujii established a transversality condition, which is sufficient to obtain the absolute continuity of $\vartheta$. Let
\[
\Sigma^*=\bigcup_{n=0}^{\infty}\left\{0,\ldots, b-1\right\}^n
\]
be the set of finite length words of symbols. For a finite length word $(i_1,\dots,i_n) \in \Sigma^*$ let $[i_1,\dots,i_n]$ be the corresponding cylinder set, i.e.
\[
[i_1,\dots,i_n]=\left\{(j_1,j_2,\dots)\in\Sigma: j_1=i_1,\dots, j_n=i_n\right\}.
\]

\begin{definition}[Tsujii, \cite{T}]\label{def_trans} Fix an integer $b \ge 2$ and $\gamma \in (1/b, 1)$. Let $\varepsilon, \delta > 0$, $\ii, \jj \in \Sigma$, $m \in \N$, $k \in \{1,\ldots, b^m\}$. The functions $S(\cdot, \ii)$ and $S(\cdot,\jj)$ are called $(\varepsilon, \delta)$-transversal on the interval $I_{m,k} = [(k-1)/b^m, k/b^m]$ if for every $x \in I_{m,k}$,
\[
\left|S(x,\ii)-S(x,\jj)\right| > \varepsilon \quad \text{or} \quad \left|\frac{d}{dx}S(x,\ii)-\frac{d}{dx}S(x,\jj)\right|>\delta.
\]
Otherwise they are called $(\varepsilon, \delta)$-tangent on $I_{m,k}$.

Let $\e(n, m; \varepsilon, \delta)$ be the maximum over $k \in \{1,\ldots, b^m\}$ and $(i_1, \ldots, i_n) \in \Sigma^*$ of the maximal number of finite words $(j_1, \ldots, j_n) \in \Sigma^*$ for which there exist $\ii \in [i_1, \ldots, i_n]$ and $\jj \in [j_1, \ldots, j_n]$ such that the functions $S(\cdot,\ii)$ and $S(\cdot,\jj)$ are $(\varepsilon, \delta)$-tangent on $I_{m,k}$.
\end{definition}

In \cite{T}, Tsujii proved the following result.

\begin{prop}[Tsujii, \mbox{\cite[Proposition~8]{T}}]\label{tsujii}
If $\e(n, m; \varepsilon, \delta) <\gamma^n b^n$ for some $\varepsilon, \delta > 0$ and positive integers $n, m$, then the SBR measure $\vartheta$ for $T$ is absolutely continuous.
\end{prop}

To prove Theorems~A and~B, we apply Tsujii's results for $\gamma = 1/(b\lambda)$ and $\psi = \phi'$, where $\phi$ is a $\mathbb Z$-periodic $C^3$ function. In this case there is a direct relation between the measure $\vartheta$ and the measures $m_{x,\gamma}$ defined in \eqref{def_m}. More precisely, by \eqref{eproj} we have
\begin{equation}\label{YS}
Y_{x,\gamma}(\ii) = -\gamma S(x,\ii),
\end{equation}
so \eqref{nu} gives
\[
\vartheta(A) =
(\mathcal L|_{\mathbb S^1} \times \mathbb P)\left(\left\{(x, \ii): \left(x, -\frac{Y_{x,\gamma}(\ii)}{\gamma}\right) \in A\right\}\right)
= \int_{\mathbb S^1} m_{x,\gamma} (\{-\gamma y: (x, y) \in A\}) dx
\]
for every Lebesgue measurable set $A \subset \mathbb S^1 \times \mathbb R$.
This easily implies the following lemma.
\begin{lemma}\label{mu}
If the SBR measure $\vartheta$ for $T(x,y) = (bx, \gamma y + \phi' (x))$ is absolutely continuous, then the measure $m_{x,\gamma}$ is absolutely continuous for Lebesgue almost every $x\in(0,1)$, in particular $\dim m_{x,\gamma} = 1$ for Lebesgue almost every $x\in(0,1)$.
\end{lemma}

Together with Corollary~\ref{cor_led}, this gives the following.

\begin{cor}\label{cor:main} Let $b \geq 2$ be an integer, $\lambda \in (1/b, 1)$ and let $\phi$ be a $\mathbb Z$-periodic $C^3$ function. If the SBR measure $\vartheta$ for $T(x,y) = (bx, \gamma y + \phi' (x))$, where $\gamma = 1/(b\lambda)$, is absolutely continuous, then
\[
\dim \mu_{\lambda, b}^{\phi} = D.
\]
Moreover, if $(\gamma, \phi') \in \mathcal D^\circ$, then $(\lambda, \phi) \in \interior \mathcal{F}_b$.
\end{cor}

Now we are ready to prove Theorem~A.

\begin{proof}[Proof of Theorem~{\rm A}]
Fix an integer $b \ge 2$, $\lambda \in (1/b, 1)$ and let $\gamma = 1/(b\lambda)$.
Consider the functions $\psi_1, \ldots, \psi_m$ from Theorem~\ref{tsujiicor}, treated as $\mathbb{Z}$-periodic functions on $\mathbb R$ and define
\[
\phi_j(x) = \int_0^x \psi_j(t)dt - x \int_0^1 \psi_j(t) dt, \qquad j = 1, \ldots, m
\]
for $x \in \mathbb R$. Then $\phi_j$ are $\mathbb{Z}$-periodic $C^\infty$ functions and can be regarded as elements of $C^\infty(\mathbb S^1, \mathbb R)$.

Take $\phi \in C^3(\mathbb S^1, \mathbb R)$ and $(t_1, \ldots, t_m) \in \mathbb{R}^m$. Then
\[
\left(\phi + \sum_{j=1}^m t_j \phi_j\right)' = \phi' + \tilde\psi + c,
\]
where
\[
\tilde\psi = \sum_{j=1}^m t_j \psi_j
\]
and $c$ is a real constant given by
\[
c = - \sum_{j=1}^m t_j \int_0^1 \psi_j(t) dt.
\]
Since $\phi' \in C^2(\mathbb S^1, \mathbb R)$, by Theorem~\ref{tsujiicor} we have
\[
\left(\gamma, \phi' + \tilde\psi\right) \in \mathcal D^\circ
\]
for Lebesgue almost every $(t_1, \ldots, t_m)$.
In view of Corollary~\ref{cor:main}, to end the proof of Theorem~A it is sufficient to use the following observation.
\begin{equation}\label{psi+c}
\text{If} \quad (\gamma, \psi) \in \mathcal D, \quad \text{then} \quad (\gamma, \psi + c)  \in \mathcal D \text{ for every } c \in \mathbb R.
\end{equation}
To show \eqref{psi+c}, consider $S(x, \ii)$, $\Psi$ and $\vartheta$ for a function $\psi$ with $(\gamma, \psi) \in \mathcal D$ and denote the corresponding objects for the function $\psi + c$, $c \in \mathbb{R}$, by $S^c(x, \ii)$, $\Psi^c$ and $\vartheta^c$. By definition,
\[
S^c(x, \ii) = S(x, \ii) + c \sum_{n=1}^\infty \gamma^{n-1} = S(x, \ii) + \frac{c}{1-\gamma},
\]
so $\Psi^c$ is a composition of $\Psi$ with the translation $(x,y) \mapsto (x, y + c/(1-\gamma))$ on $\mathbb S^1 \times \mathbb R$. Hence, \eqref{nu} implies immediately that if $\vartheta$ is absolutely continuous, then the same holds for $\vartheta^c$, which proves \eqref{psi+c}.
\end{proof}

\section{Proof of Theorem~{\rm B}}\label{sec:thmB}

For the rest of the paper, we assume
\[
\phi(x)=\cos(2\pi x).
\]

To use Corollary~\ref{cor:main} and, consequently, prove Theorem~B, for an integer $b \ge 2$ we find conditions on $\lambda \in (1/b, 1)$ under which the SBR measure $\vartheta$ for the map $T$ defined in Section~\ref{sec:thmA}, with $\gamma = 1/(b\lambda)$ and $\psi(x) = \phi'(x) = -2\pi\sin(2\pi x)$, is absolutely continuous. Note that \eqref{eproj} has the form
\begin{equation}
\label{eproj'}
Y_{x,\gamma}(\ii)=2\pi \sum_{n=1}^{\infty}\gamma^n\sin \left(2\pi\left(\frac{x}{b^n}+\frac{i_1}{b^n}+\cdots+\frac{i_n}{b}\right)\right).
\end{equation}

To use Proposition~\ref{tsujii}, we check the transversality condition for the functions $Y_{\cdot,\gamma}$ (by \eqref{YS}, this is equivalent to the transversality for the functions $S(\cdot,\ii)$). First, we prove the existence of the numbers $\lambda_b$ defined in Theorem~B.

\begin{lemma}\label{h} For every integer $b \ge 2$, the function $h_b$ is strictly decreasing on the interval $(1/b, 1)$ and has a unique zero $\lambda_b \in (1/b, 1)$. In particular, $\lambda_2 < 0.9352, \lambda_3 < 0.7269$, $\lambda_4 < 0.6083$ and $\lambda_b < 0.5448$ for $b \geq 5$. Moreover, $\lambda_b \to 1/\pi$ as $b \to \infty$.
\end{lemma}
\begin{proof}
Consider first the case $b = 2$. We easily check
\[
\frac{d}{d \lambda}\left(-\frac{1}{8\lambda^2} + \frac{\sqrt{2}}{2\lambda}\right) < 0
\]
for $\lambda \in (1/2, 1)$, which immediately implies that the function $h_2$ is strictly decreasing on the interval $(1/2, 1]$. Moreover, $h_2(\lambda) \to
+\infty$ as $\lambda \to (1/2)^+$ and $h_2(1) < 0$. Hence, $h_2$ has a unique zero $\lambda_2 \in (1/2, 1)$.

Consider now the case $b \geq 3$.
It is obvious that $h_b$ is strictly decreasing on the interval $(1/b, 1]$ and tends to $+\infty$ as $\lambda \to (1/b)^+$. Using the inequality $\sin x > x - x^3/6$ for $x > 0$, we get
\[
h_b(\lambda) < \frac{1}{(b\lambda - 1)^2} + \frac{1}{(b^2\lambda - 1)^2} + \frac{\pi^4}{3b^4}- \frac{\pi^2}{b^2} = \frac{H_b(\lambda)}{b^2}
\]
for
\[
H_b(\lambda) = \frac{1}{(\lambda - 1/b)^2} + \frac{1}{(b\lambda - 1/b)^2} + \frac{\pi^4}{3b^2} - \pi^2.
\]
For $\lambda \in (1/b, 1]$, the function $b \mapsto H_b(\lambda)$ is strictly decreasing. Moreover, $H_3(1) < 0$, so 
\begin{equation}\label{h_b(1)}
h_b(1) < 0 \quad \text{for } b\geq 3. 
\end{equation}
This proves the existence of the unique zero $\lambda_b \in (1/b, 1)$ of the function $h_b$.

One can directly check that $h_2(0.9352), \; h_3(0.7269), \; h_4(0.6083) < 0$, which shows $\lambda_2 < 0.9352$, $\lambda_3 < 0.7269$, $\lambda_4 < 0.6083$. Moreover,
$H_5(0.5448) < 0$, so $H_b(0.5448) < 0$ for every $b \geq 5$, which implies $\lambda_b < 0.5448$ for $b \geq 5$.
The last assertion of the lemma follows easily from the definition of the function $h_b$ and the fact $\lim_{x\to 0} \sin x /x = 1$.
\end{proof}

Now we prove the transversality condition for the functions $Y_{\cdot,\gamma}$.

\begin{prop}\label{delta}
If $\gamma \in (1/b, 1/(b\lambda_b))$, then there exists $\delta > 0$ such that for every $\ii = (i_1, i_2, \ldots),\: \jj= (j_1, j_2, \ldots) \in \Sigma$ with $i_1 \neq j_1$ and every $x \in [0,1]$,
\[
\left|Y_{x,\gamma}(\ii)-Y_{x,\gamma}(\jj)\right| > \delta \quad \text{or} \quad \left|\frac{d}{dx}Y_{x,\gamma}(\ii)-\frac{d}{dx}Y_{x,\gamma}(\jj)\right|>\delta.
\]
\end{prop}
\begin{proof}
Fix $\gamma \in (1/b, 1/(b\lambda_b))$. Suppose the assertion does not hold. Then for every $\delta > 0$ there exist $\ii = (i_1, i_2, \ldots),\: \jj= (j_1, j_2, \ldots) \in \Sigma$ with $i_1 \neq j_1$ and $x \in [0,1]$, such that
\begin{equation}\label{<delta}
\left|Y_{x,\gamma}(\ii)-Y_{x,\gamma}(\jj)\right| \leq \delta, \quad  \left|\frac{d}{dx}Y_{x,\gamma}(\ii)-\frac{d}{dx}Y_{x,\gamma}(\jj)\right| \leq \delta.
\end{equation}
First, consider the case $b \geq 3$. By the definition of $Y_{x,\gamma}$ (see \eqref{eproj'}),
\begin{equation}\label{Y>}
\begin{split}
\left|Y_{x,\gamma}(\ii)-Y_{x,\gamma}(\jj)\right| &\geq 2\pi\gamma\left|\sin\left(2\pi\frac{x+i_1}{b}\right) - \sin\left(2\pi\frac{x+j_1}{b}\right)\right| - 4\pi\sum_{n=2}^\infty \gamma^n\\
&=4\pi\gamma\sin\left(2\pi\frac{|i_1-j_1|}{2b}\right) \left|\cos\left(2\pi\frac{2x+i_1+j_1}{2b}\right)\right| - \frac{4\pi\gamma^2}{1-\gamma}\\
&\geq 4\pi\gamma\sin\frac{\pi}{b} \left|\cos\left(2\pi\frac{2x+i_1+j_1}{2b}\right)\right| - \frac{4\pi\gamma^2}{1-\gamma},
\end{split}
\end{equation}
as $1 \leq |i_1-j_1| \leq b-1$. Similarly, since
\[
\frac{d}{dx}Y_{x,\gamma}(\ii)=4\pi^2\sum_{n=1}^{\infty}\left(\frac{\gamma}{b}\right)^n\cos\left(2\pi\left(\frac{x}{b^n}+\frac{i_1}{b^n}+\cdots+\frac{i_n}{b}\right)\right),
\]
we obtain
\begin{equation}\label{dY>}
\begin{split}
\left|\frac{d}{dx}Y_{x,\gamma}(\ii)-\frac{d}{dx}Y_{x,\gamma}(\jj)\right|&\geq \frac{4\pi^2 \gamma}{b}\left|\cos\left(2\pi\frac{x+i_1}{b}\right)-\cos\left(2\pi\frac{x+j_1}{b}\right)\right|-8\pi^2\sum_{n=2}^{\infty}\left(\frac{\gamma}{b}\right)^n\\
&=\frac{8\pi^2\gamma}{b}\sin\left(2\pi\frac{|i_1-j_1|}{2b}\right) \left|\sin\left(2\pi\frac{2x+i_1+j_1}{2b}\right)\right| - \frac{8\pi^2\gamma^2}{b(b-\gamma)}\\
&\geq \frac{8\pi^2\gamma}{b}\sin\frac{\pi}{b} \left|\sin\left(2\pi\frac{2x+i_1+j_1}{2b}\right)\right| - \frac{8\pi^2\gamma^2}{b(b-\gamma)}.
\end{split}
\end{equation}
By \eqref{<delta}, \eqref{Y>} and \eqref{dY>},
\begin{align*}
\sin\frac{\pi}{b} \left|\cos\left(2\pi\frac{2x+i_1+j_1}{2b}\right)\right| &\leq \frac{\gamma}{1-\gamma} + \frac{\delta}{4\pi \gamma},\\
\sin\frac{\pi}{b} \left|\sin\left(2\pi\frac{2x+i_1+j_1}{2b}\right)\right| &\leq \frac{\gamma}{b-\gamma} + \frac{\delta b}{8\pi^2 \gamma}.
\end{align*}
Taking the sum of the squares of the two inequalities, we get
\[
\sin^2\frac{\pi}{b} \leq \left(\frac{\gamma}{1-\gamma} + \frac{\delta}{4\pi \gamma}\right)^2 + \left( \frac{\gamma}{b-\gamma} + \frac{\delta b}{8\pi^2 \gamma}\right)^2.
\]
Since $\delta$ is arbitrarily small, in fact this implies
\[
0 \leq \frac{\gamma^2}{(1-\gamma)^2} + \frac{\gamma^2}{(b-\gamma)^2} - \sin^2\frac{\pi}{b} = h_b(\lambda)
\]
for $\lambda = 1/(b\gamma) > \lambda_b$, which contradicts Lemma~\ref{h}. This ends the proof in the case $b \geq 3$.

Consider now the case $b= 2$. We improve the estimates made by Tsujii in \cite[Appendix]{T}. In this case we need to consider also the second term of
$Y_{x,\gamma}$. Since $i_1 \neq j_1$, we can assume $i_1 = 1$, $j_1 = 0$. Then
\begin{align*}
&\left|Y_{x,\gamma}(\ii)-Y_{x,\gamma}(\jj)\right|\\
&\geq 2\pi\gamma\left|\sin(\pi(x+1)) - \sin(\pi x) + \gamma\left(\sin\left(\pi\frac{x+1+2i_2}{2}\right) - \sin\left(\pi\frac{x +2j_2}{2}\right)\right)\right| - 4\pi\sum_{n=3}^\infty \gamma^n\\
&=4\pi\gamma\left| \sin(\pi x)
- \gamma\left(\sin\left(\pi\frac{1 + 2(i_2-j_2)}{4}\right) \cos\left(\pi\frac{2x+1+2(i_2 + j_2)}{4}\right)\right)
\right| - \frac{4\pi\gamma^3}{1-\gamma}
\end{align*}
and
\begin{align*}
&\left|\frac{d}{dx}Y_{x,\gamma}(\ii)-\frac{d}{dx}Y_{x,\gamma}(\jj)\right|\\
&\geq 2\pi^2 \gamma\left|\cos(\pi(x+1))-\cos(\pi x)
+ \frac{\gamma}{2}\left(\cos\left(\pi\frac{x+1+2i_2}{2}\right)-\cos\left(\pi\frac{x+2j_2}{2}\right)\right)
\right|-8\pi^2\sum_{n=3}^{\infty}\left(\frac{\gamma}{2}\right)^n\\
&=4\pi^2\gamma\left| \cos(\pi x)
+ \frac{\gamma}{2}\left(\sin\left(\pi \frac{1+2(i_2-j_2)}{4}\right) \sin\left(\pi\frac{2x+1 + 2(i_2 + j_2)}{4}\right)\right)
\right| - \frac{2\pi^2\gamma^3}{2-\gamma}
\end{align*}
which together with \eqref{<delta} implies
\begin{align*}
\left| \sin(\pi x)
- \gamma\left(\sin\left(\pi\frac{1 + 2(i_2-j_2)}{4}\right) \cos\left(\pi\frac{2x+1+2(i_2 + j_2)}{4}\right)\right)
\right| &\leq \frac{\gamma^2}{1-\gamma} + \frac{\delta}{4\pi \gamma},\\
\left| \cos(\pi x)
+ \frac{\gamma}{2}\left(\sin\left(\pi \frac{1+2(i_2-j_2)}{4}\right) \sin\left(\pi\frac{2x+1 + 2(i_2 + j_2)}{4}\right)\right)
\right|&\leq \frac{\gamma^2}{2(2-\gamma)} + \frac{\delta}{4\pi^2 \gamma}.
\end{align*}
Recall that $i_2, j_2, x$ depend on $\delta$. Taking a sequence of $\delta$-s tending to $0$ we can choose a subsequence such that $i_2, j_2, x$ converge, so by continuity we can assume
\begin{align*}
\left| \sin(\pi x)
- \gamma\left(\sin\left(\pi\frac{1 + 2(i_2-j_2)}{4}\right) \cos\left(\pi\frac{2x+1+2(i_2 + j_2)}{4}\right)\right)
\right| &\leq \frac{\gamma^2}{1-\gamma},\\
\left| \cos(\pi x)
+ \frac{\gamma}{2}\left(\sin\left(\pi \frac{1+2(i_2-j_2)}{4}\right) \sin\left(\pi\frac{2x+1 + 2(i_2 + j_2)}{4}\right)\right)
\right|&\leq \frac{\gamma^2}{2(2-\gamma)}.
\end{align*}
for some $i_2, j_2 \in \{0,1\}$ and $x \in [0,1]$. Taking the sum of the squares of the two inequalities and noting that $\sin^2(\pi(1+2(i_2-j_2))/4) = 1/2$, we obtain
\begin{equation}\label{g>}
g(x) \geq 0,
\end{equation}
where
\[
g(t) = \tilde g(t) - \frac{3\gamma^2}{8} \cos^2\left(\pi\frac{2t+1+2(i_2 + j_2)}{4}\right)
\]
for
\begin{align*}
\tilde g(t) &= \frac{\gamma^4}{(1-\gamma)^2} +  \frac{\gamma^4}{4(2-\gamma)^2}
- \frac{\gamma^2}{8}- 1\\ &+ 2\gamma \sin\left(\pi\frac{1 + 2(i_2-j_2)}{4}\right)\sin(\pi t)\cos\left(\pi\frac{2t+1+2(i_2 + j_2)}{4}\right)\\ &- \gamma \sin\left(\pi\frac{1 + 2(i_2-j_2)}{4}\right)\cos(\pi t)\sin\left(\pi\frac{2t+1+2(i_2 + j_2)}{4}\right).
\end{align*}
We have
\begin{multline*}
g'(t) = \frac{3\pi \gamma}{8} \cos\left(\pi\frac{2t+1+2(i_2 + j_2)}{4}\right)\\\left(4\sin\left(\pi\frac{1 + 2(i_2-j_2)}{4}\right)\cos(\pi t) + \gamma\sin\left(\pi\frac{2t+1+2(i_2 + j_2)}{4}\right)\right)
\end{multline*}
and
\[
\tilde g'(t) = \frac{3\pi \gamma}{2}\sin\left(\pi\frac{1 + 2(i_2-j_2)}{4}\right)\cos(\pi t) \cos\left(\pi\frac{2t+1+2(i_2 + j_2)}{4}\right).
\]
Now we consider four cases, depending on the values of $i_2, j_2$.

First, let $i_2= j_2 = 0$. Then
\[
\tilde g'(t) = \frac{3\sqrt{2}\pi \gamma}{4}\cos(\pi t) \cos\left(\pi\frac{2t+1}{4}\right) \geq 0
\]
for $t \in [0,1]$. Hence,
\begin{equation}\label{case1}
g(x) \leq \tilde g(x) \leq \tilde g(1) = \frac{\gamma^4}{(1-\gamma)^2} +  \frac{\gamma^4}{4(2-\gamma)^2}
- \frac{\gamma^2}{8}+ \frac{\gamma}{2}- 1.
\end{equation}

Let now $i_2= j_2 = 1$. Then
\[
\tilde g'(t) = -\frac{3\sqrt{2}\pi \gamma}{4}\cos(\pi t) \cos\left(\pi\frac{2t+1}{4}\right) \leq 0
\]
for $t \in [0,1]$, so
\begin{equation}\label{case2}
g(x) \leq \tilde g(x) \leq \tilde g(0) = \frac{\gamma^4}{(1-\gamma)^2} +  \frac{\gamma^4}{4(2-\gamma)^2}
- \frac{\gamma^2}{8}+ \frac{\gamma}{2}- 1.
\end{equation}

The third case is $i_2=1$, $j_2 = 0$. Then
\[
g'(t) = -\frac{3\pi \gamma}{8} \sin\left(\pi\frac{2t+1}{4}\right)\\\left(2 \sqrt{2}\cos(\pi t) + \gamma\cos\left(\pi\frac{2t+1}{4}\right)\right)
\begin{cases}
\leq 0 &\text{for } t \in [0,1/2]\\
> 0 &\text{for } t \in (1/2,1],
\end{cases}
\]
which implies
\begin{equation}\label{case3}
g(x) \leq \max(g(0), g(1)) = \frac{\gamma^4}{(1-\gamma)^2} +  \frac{\gamma^4}{4(2-\gamma)^2}
- \frac{5\gamma^2}{16}- \frac{\gamma}{2}- 1.
\end{equation}

The last case is $i_2=0$, $j_2 = 1$. Then
\begin{align*}
g'(t) &= -\frac{3\pi \gamma}{8} \sin\left(\pi\frac{2t+1}{4}\right)\left(-2 \sqrt{2}\cos(\pi t) + \gamma\cos\left(\pi\frac{2t+1}{4}\right)\right)\\
&= -\frac{3\sqrt{2}\pi \gamma}{16}\sin\left(\pi\frac{2t+1}{4}\right)\left(\cos\frac{\pi t}{2} - \sin\frac{\pi t}{2}\right)\left(\gamma - 4\left(\cos\frac{\pi t}{2} + \sin\frac{\pi t}{2}\right)\right)\\
&\begin{cases}
\geq 0 &\text{for } t \in [0,1/2]\\
< 0 &\text{for } t \in (1/2,1],
\end{cases}
\end{align*}
since $\gamma - 4(\cos(\pi t/2) + \sin(\pi t/2)) \leq \gamma - 4 < 0$ for $t \in [0,1]$. Hence,
\begin{equation}\label{case4}
g(x) \leq g(1/2) = \frac{\gamma^4}{(1-\gamma)^2} +  \frac{\gamma^4}{4(2-\gamma)^2}
- \frac{\gamma^2}{2} + \sqrt{2} \gamma- 1.
\end{equation}

Considering the conditions \eqref{case1}--\eqref{case4} we easily conclude that the largest upper estimate for $g(x)$ appears in \eqref{case4}. Therefore, by \eqref{g>}, in all cases we have
\[
0 \leq \frac{\gamma^4}{(1-\gamma)^2} +  \frac{\gamma^4}{4(2-\gamma)^2}
- \frac{\gamma^2}{2} + \sqrt{2} \gamma - 1 = h_2(\lambda)
\]
for $\lambda = 1/(2\gamma) > \lambda_2$, which contradicts Lemma~\ref{h}. This ends the proof in the case $b=2$.
\end{proof}

To conclude the proof of Theorem~B, it is enough to notice that by Proposition~\ref{delta} and \eqref{YS}, for $\lambda \in (\lambda_b, 1)$ we have $\e(1,1; \delta/\gamma, \delta/\gamma) = 1 < \gamma b$ for the functions $S(\cdot,\ii)$ (see Definition~\ref{def_trans}) and use Proposition~\ref{tsujii} and Corollary~\ref{cor:main}. The estimates for $\lambda_2$, $\lambda_3$ and $\lambda_4$ in Corollary~C follow directly from Lemma~\ref{h}.

\section{Proof of Theorem {\rm C}}\label{sec:thmC}

Using the transversality method developed by Peres and Solomyak in the study of infinite Bernoulli convolutions (see \cite{peso,peso2}), with a minor modification on the standard argument, we show that the measure $m_{x,\gamma}$ is absolutely continuous for Lebesgue almost every $(x,\gamma)\in(0,1)\times(1/b,1/(b\tilde\lambda_b))$. Then Theorem~C will follow by the Fubini Theorem and Corollary~\ref{cor_led}.

First, we prove the existence of the numbers $\tilde \lambda_b$ defined in Theorem~C.

\begin{lemma}\label{tilde}
For every integer $b \ge 2$ there exists a unique number $\tilde \lambda_b \in (1/b, 1)$ such that
\begin{equation}\label{etr}
y\left(\frac{1}{\sqrt{\sin^2(\pi/b) - 1/(b^2\tilde \lambda_b - 1)^2}}\right) = \frac{1}{b\tilde \lambda_b}
\end{equation}
and for $\lambda \in (1/b, 1)$,
\[
y\left(\frac{1}{\sqrt{\sin^2(\pi/b) - 1/(b^2\lambda - 1)^2}}\right) < \frac{1}{b\lambda} \quad \iff \lambda \in (1/b, \tilde \lambda_b).
\]
Moreover, $\tilde\lambda_b < \lambda_b$ for every $b \geq 2$, $\tilde \lambda_b < 1.04/\sqrt{b}$ for every $b \geq 5$ and $\tilde\lambda_b \sqrt{b} \to 1/\sqrt{\pi}$ as $b \to \infty$.
\end{lemma}
\begin{proof} First, note that
\begin{equation}\label{root}
\sin^2 \frac{\pi}{b} - \frac{1}{(b^2\lambda - 1)^2} > 0 \quad \text{for every }\lambda \in (1/b, 1).
\end{equation}
Indeed, for $b = 2$ it is obvious and for $b \geq 3$,
\[
\sin^2 \frac{\pi}{b} - \frac{1}{(b^2\lambda - 1)^2} > \sin^2 \frac{\pi}{b} - \frac{1}{(b - 1)^2} > \sin^2 \frac{\pi}{b} - \frac{1}{(b - 1)^2} - \frac{1}{(b^2 - 1)^2} = -h_b(1) > 0
\]
by \eqref{h_b(1)}. In particular, this implies that
\[
\beta = \beta (\lambda) = \frac{1}{\sqrt{\sin^2(\pi/b) - 1/(b^2\lambda - 1)^2}}
\]
is well-defined for $\lambda \in (1/b, 1)$. Obviously, $\beta > 1$.

It is known (see \cite{peso2}) that for $\beta \geq 1$ the function $\beta \mapsto y(\beta)$ is continuous and strictly decreasing. This implies that the function $\lambda \mapsto y(\beta(\lambda))-1/(b\lambda)$ is continuous and strictly increasing on $(1/b, 1)$. Moreover (see \cite[Corollary~5.2]{peso2}), 
$y(\beta)$ satisfies
\begin{equation}\label{y(2)}
y(2) = \frac 1 2,
\end{equation}
\begin{equation}\label{y>}
1 > y(\beta) \geq \frac{1}{1+ \sqrt{\beta}},
\end{equation}
and
\begin{equation}\label{y=}
y(\beta) = \frac{1}{1+ \sqrt{\beta}} \qquad \text{for } \beta \geq 3+ \sqrt{8}.
\end{equation}

Consider first the case $b = 2$. Then $\beta(\lambda) \to +\infty$ as $ \lambda \to (1/2)^+$, so by \eqref{y>}, $y(\beta(\lambda)) \to 0$ as $ \lambda \to (1/2)^+$ and hence $y(\beta(\lambda)) - 1 /(2\lambda) < 0$ for $\lambda$ close to $1/2$. On the other hand, by \eqref{y(2)},
\[
y(\beta(1)) - \frac 1 2  = y\left(\frac{3\sqrt 2}{4}\right)  - \frac 1 2 > y(2) - \frac 1 2  = 0,
\]
which shows the existence of a unique number $\tilde\lambda_2 \in (1/2, 1)$ satisfying \eqref{etr}. To see that $\tilde\lambda_2 < \lambda_2$, we directly check that $h_2(0.9) > 0$ (which gives $\lambda_2 > 0.9$) and $\tilde\lambda_2 < 0.81$. The latter estimate is done at the end of the paper by obtaining a suitable $(*)$-function and using \cite[Lemma~5.1]{peso2} (see Lemma~\ref{lpeso}).

Consider now the case $b \ge 3$. As previously, we have $y(\beta(\lambda)) - 1/(b\lambda) < 0$ for $\lambda$ close to $1/b$. Moreover, by \eqref{y>},
\begin{equation}\label{beta}
y(\beta(\lambda_b)) - \frac{1}{b\lambda_b} \geq \frac{1}{1+\sqrt{\beta(\lambda_b)}} - \frac{1}{b\lambda_b}.
\end{equation}
By the definition of $\beta$, the inequality
\begin{equation}\label{>0}
\frac{1}{1+\sqrt{\beta(\lambda_b)}} - \frac 1 {b\lambda_b} > 0
\end{equation}
is equivalent to $\tilde h_b (\lambda_b) < 0$ for
\[
\tilde h_b (\lambda) = \frac{1}{(b\lambda - 1)^4} + \frac{1}{(b^2\lambda - 1)^2} - \sin^2 \frac{\pi}{b}.
\]
We have $\tilde h_b (\lambda) < h_b(\lambda)$ for $\lambda \in (2/b, 1)$ and
\[
h_b\left(\frac 2 b\right) = 1 + \frac{1}{(2b-1)^2} - \sin^2 \frac \pi b > 0,
\]
so by Lemma~\ref{h}, $\lambda_b > 2/b$ and hence $\tilde h_b (\lambda_b) < h_b(\lambda_b) = 0$, which shows \eqref{>0}. By \eqref{beta}, this implies $y(\beta(\lambda_b)) - 1/(b\lambda_b) > 0$, so there exists a unique number $\tilde \lambda_b \in (1/b, 1)$ such that $\tilde\lambda_b < \lambda_b$ and $\tilde\lambda_b$ satisfies \eqref{etr}.  

Like in the proof of Lemma~\ref{h}, using the inequality $\sin x - x^3/6$ for $x > 0$, we obtain
\[
\tilde h_b (\lambda) < \frac{1}{(b\lambda - 1)^4} + \frac{1}{(b^2\lambda - 1)^2} + \frac{\pi^4}{3b^4}- \frac{\pi^2}{b^2} = \frac{\tilde H_b(\lambda)}{b^2}
\]
for
\[
\tilde H_b(\lambda) = \frac{1}{(\sqrt{b} \lambda - 1/\sqrt{b})^4} + \frac{1}{(b\lambda - 1/b)^2} + \frac{\pi^4}{3b^2} - \pi^2.
\]
Substituting $\lambda = c/\sqrt{b}$ for $c > 0$, we get
\[
\tilde H_b(c/\sqrt{b}) = \frac{1}{(c - 1/\sqrt{b})^4} + \frac{1}{(c\sqrt{b} - 1/b)^2} + \frac{\pi^4}{3b^2} - \pi^2.
\]
The function $\tilde H_b(c/\sqrt{b})$ is strictly decreasing with respect to $c$ and $b$ and one can directly check $\tilde H_5(1.04/\sqrt{5}) < 0$. This implies that $\tilde \lambda_b < 1.04/\sqrt{b}$ for every $b \geq 5$.

For $\beta \geq 19$,
\[
\beta > \frac{1}{\sin(\pi/19)} > \frac{19}{\pi} > 3 + \sqrt{8},
\]
so by \eqref{y=}, the number $\tilde\lambda_b$ is equal to the unique zero of the function $\tilde h_b$ on the interval $(1/b, 1)$. This easily implies that $\tilde\lambda_b \sqrt{b} \to 1/\sqrt{\pi}$ as $b \to \infty$ (the details are left to the reader).
\end{proof}

Let
\begin{equation}\label{tildegamma}
\tilde\gamma_b = \frac{1}{b\tilde\lambda_b}.
\end{equation}
Now we prove a modified transversality condition for the functions
$Y_{\cdot,\cdot}(\ii)$. The trick we use is to consider transversality with respect to two variables $x, \gamma$.

\begin{prop}\label{ltrans} For every $\varepsilon > 0$ there exists $\delta>0$ such that for every $\ii = (i_1, i_2, \ldots), \:\jj = (j_1, j_2, \ldots) \in\Sigma$ with $i_1\neq j_1$,
\[
\left|Y_{x,\gamma}(\ii)-Y_{x,\gamma}(\jj)\right|> \delta \quad \text{or} \quad \left|\frac{d}{dx}Y_{x,\gamma}(\ii)-\frac{d}{dx}Y_{x,\gamma}(\jj)\right|+\left|\frac{d}{d\gamma}Y_{x,\gamma}(\ii)-\frac{d}{d\gamma}Y_{x,\gamma}(\jj)\right|>\delta
\]
for every $x \in (0,1)$ and $\gamma \in (1/b + \varepsilon, \tilde\gamma_b - \varepsilon)$.
\end{prop}

\begin{proof} The proof is similar to the proof of Proposition~\ref{delta}. 
Note first that by \eqref{root} and \eqref{gamma}, 
\begin{equation}\label{root2}
\sin^2\frac{\pi}{b} - \left(\frac{\gamma}{b-\gamma}\right)^2 > 0 \quad \text{for } \gamma \in (1/b, 1).
\end{equation}
Suppose that the assertion of the proposition does not hold for some $\varepsilon > 0$. Then for every $\delta > 0$ there exist
$\ii = (i_1, i_2, \ldots), \:\jj = (j_1, j_2, \ldots) \in\Sigma$ with $i_1\neq j_1$, $x \in (0, 1)$ and $\gamma \in (1/b + \varepsilon, \tilde\gamma_b - \varepsilon)$, such that
\begin{equation}\label{eb1}
\left|Y_{x,\gamma}(\ii)-Y_{x,\gamma}(\jj)\right| \leq \delta, \quad \left|\frac{d}{dx}Y_{x,\gamma}(\ii)-\frac{d}{dx}Y_{x,\gamma}(\jj)\right| \leq \delta, \quad \left|\frac{d}{d\gamma}Y_{x,\gamma}(\ii)-\frac{d}{d\gamma}Y_{x,\gamma}(\jj)\right| \leq\delta.
\end{equation}
Repeating the estimates in \eqref{dY>}, we obtain
\begin{equation}\label{tildedY>}
\left|\frac{d}{dx}Y_{x,\gamma}(\ii)-\frac{d}{dx}Y_{x,\gamma}(\jj)\right|\geq  \frac{8\pi^2\gamma}{b}\sin\frac{\pi}{b} \left|\sin\left(2\pi\frac{2x+i_1+j_1}{2b}\right)\right| - \frac{8\pi^2\gamma^2}{b(b-\gamma)}.
\end{equation}
By \eqref{eb1} and \eqref{tildedY>},
\begin{equation}\label{sin<}
\sin\frac{\pi}{b} \left|\sin\left(\frac{\pi(2x+i_1+j_1)}{b}\right)\right| \leq \frac{\gamma}{b-\gamma} + \frac{\delta b}{8\pi^2 \gamma} < \frac{\gamma}{b-\gamma} + \frac{\delta b^2}{8\pi^2}.
\end{equation}
By the definition of $Y_{x,\gamma}$ (see \eqref{eproj'}), we have
\[
Y_{x,\gamma}(\ii)-Y_{x,\gamma}(\jj)= 2\pi \sum_{n=1}^\infty y_n \gamma^n,
\]
where
\[
y_1 = \sin\left(2\pi\frac{x+i_1}{b}\right) - \sin\left(2\pi\frac{x+j_1}{b}\right) = 2\sin\left(2\pi\frac{i_1-j_1}{2b}\right) \cos\left(2\pi\frac{2x+i_1+j_1}{2b}\right)
\]
and 
\[
|y_n| \leq 2 \quad \text{for } n \geq 2.
\]
Using the fact $i_1 \neq j_1$, we obtain
\begin{equation}\label{y_1'}
|y_1| \geq 2\sin\frac{\pi}{b} \left|\cos\left(2\pi\frac{2x+i_1+j_1}{2b}\right)\right|.
\end{equation}
By \eqref{root2} and due to the fact $\gamma \in (1/b + \varepsilon, 1 - \varepsilon)$, we have 
\[
\sin^2\frac{\pi}{b} - \left(\frac{\gamma}{b-\gamma} + \frac{\delta b^2}{8\pi^2}\right)^2 > c > 0
\]
for sufficiently small $\delta$, where $c$ does not depend on $\delta, \gamma$. Hence, using \eqref{sin<} and \eqref{y_1'}, we obtain
\begin{equation}\label{y_1}
|y_1| >2\; \sqrt{\sin^2\frac{\pi}{b} - \left(\frac{\gamma}{b-\gamma} + \frac{\delta b^2}{8\pi^2}\right)^2} > 2  \sqrt{c}
\end{equation}
for small $\delta$. Consequently, for the function
\[
g(t) = \frac{Y_{x,t}(\ii)-Y_{x,t}(\jj)}{2\pi y_1 t}
\]
we have
\[
g(t) = 1 + \sum_{n=1}^\infty g_n t^n,
\]
where
\[
|g_n| = \frac{|y_{n+1}|}{|y_1|} <
\frac{1}{\sqrt{\sin^2(\pi/b) - (\gamma/(b-\gamma) + \delta b^2/(8\pi^2))^2}}.
\]
This implies that $g \in \mathcal G_{\beta}$
for
\[
\beta = \frac{1}{\sqrt{\sin^2(\pi/b) - (\gamma/(b-\gamma) + \delta b^2/(8\pi^2))^2}}.
\]
On the other hand, by \eqref{eb1} and \eqref{y_1},
\begin{equation}\label{g<}
|g(\gamma)| \leq \frac{\delta}{2\pi |y_1| \gamma} < \frac{\delta b}{4\pi \sqrt{c}}, \qquad |g'(\gamma)| \leq \frac{(\gamma + 1)\delta}{2\pi|y_1|\gamma^2} < \frac{\delta b^2}{2\pi\sqrt{c}}.
\end{equation}
Note that $g$, $\gamma$ and $\beta$ depend on $\delta$. Take a sequence of $\delta$-s tending to $0$. Then we can choose a subsequence such that $\gamma \to \gamma_* \in [1/b + \varepsilon, \tilde\gamma_b-\varepsilon]$, $\beta \to \beta_*$ for
\[
\beta_* = \frac{1}{\sqrt{\sin^2(\pi/b) - (\gamma_*/(b-\gamma_*))^2}} < \frac{1}{\sqrt{\sin^2(\pi/b) - (\tilde\gamma_b/(b-\tilde\gamma_b))^2}}
\]
and $g$ converges uniformly in $[1/b, \tilde\gamma_b]$ to a function $g_* \in \mathcal G_{\beta_*}$.
Since by \eqref{g<}, $g(\gamma)$ and $g'(\gamma)$ tend to $0$ as $\delta \to 0$, we obtain
\[
g_*(\gamma_*) = g'_*(\gamma_*) = 0,
\]
so $y(\beta_*) \leq \gamma^*$. This is a contradiction, because by Lemma~\ref{tilde},
\[
y(\beta_*) = y\left(\frac{1}{\sqrt{\sin^2(\pi/b) - 1/(b^2\lambda_* - 1)^2}}\right) > \frac{1}{b\lambda_*} = \gamma_*
\]
for $\lambda_* = 1/(b\gamma_*) > 1/(b\tilde\gamma_b) = \tilde\lambda_b$.
This ends the proof.
\end{proof}

As a simple consequence of Proposition~\ref{ltrans} one can prove the following statement (for the proof we refer to e.g. \cite[Lemma~7.3]{SSU}).

\begin{lemma}\label{ltransleb}
For every $\varepsilon > 0$ there exists a constant $C>0$ such that for every $\ii = (i_1, i_2, \ldots, ),\: \jj = (j_1, j_2, \ldots, )\in\Sigma$ with $i_1\neq j_1$,
\[
\mathcal{L}_2\left(\left\{(x,\gamma)\in (0,1)\times (1/b+ \varepsilon, \tilde\gamma_b - \varepsilon): \left|Y_{x,\gamma}(\ii)-Y_{x,\gamma}(\jj)\right|<r\right\}\right)\leq Cr
\]
for every $r>0$, where $\mathcal{L}_2$ is the Lebesgue measure on the plane.
\end{lemma}

To state next results, we need to introduce some notation. For $\ii=(i_1,i_2,\dots)\in\Sigma$ let
\[
\left.\ii\right|_n = (i_1,\dots,i_n).
\]
For a finite word $\overline{k}\in\Sigma^*$ of length $n$ let
\[
\A_{\overline{k}}=\left\{(\ii,\jj)\in\Sigma \times \Sigma:\left.\ii\right|_n = \left.\jj\right|_n = \overline{k} \text{ and } \left.\ii\right|_{n+1} \neq \left.\jj\right|_{n+1}\right\}.
\]
We extend this definition for the empty word $\overline{k}$ setting
\[
\A_{\emptyset}=\left\{(\ii,\jj)\in\Sigma \times \Sigma :\left.\ii\right|_1 \neq \left.\jj\right|_1\right\}.
\]
For $N \geq 1$ we write
\[
\left.\A_{\overline{k}}\right|_N = \{(\ii|_N, \jj|_N): (\ii, \jj) \in
\A_{\overline{k}}\}.
\]
For $\overline{k}\in\Sigma^*$ and $\ii \in \Sigma$ we write $\overline{k}\ii$ for the standard concatenation.

For a finite length word $\overline{k} = (k_1,\dots,k_n)\in\Sigma^*$ and $x \in [0, 1]$ let
\[
v_{\overline{k}}(x)=\frac{x}{b^n}+\frac{k_1}{b^n}+\cdots+\frac{k_n}{b}.
\]
Let us observe that if $\left.\ii\right|_n = \left.\jj\right|_n = \overline{k}$ for some  $\ii,\jj\in \Sigma$, then
\begin{equation}\label{einv}
\left|Y_{x,\gamma}(\ii)-Y_{x,\gamma}(\jj)\right|=\gamma^n\left|Y_{v_{\overline{k}}(x),\gamma}(\sigma^n(\ii))-Y_{v_{\overline{k}}(x),\gamma}(\sigma^n(\jj))\right|,
\end{equation}
where $\sigma$ denotes the left-side shift on $\Sigma$.

Notice that because of the structure of the measure $m_{x,\gamma}$, it is not possible to apply directly the transversality method and Lemma~\ref{ltransleb}. To avoid this difficulty, we introduce the following lemma.

\begin{lemma}\label{ldecomp}
Let $(\ii, \jj) \in \A_{\emptyset}$. Then for every $r>0$ there exists $N=N(r)$ such that
\begin{equation}\label{eineq}
\left|Y_{x,\gamma}(\ii)-Y_{x,\gamma}(\jj)\right|<r\quad \Rightarrow \quad \left|Y_{x,\gamma}(\left.\ii\right|_N\mathbf{0})-Y_{x,\gamma}(\left.\jj\right|_N\mathbf{0})\right|<2r
\end{equation}
for every $x \in [0,1]$ and $\gamma\in (1/b, \tilde\gamma_b)$, where $\mathbf{0}=(0,0,\dots)$.
\end{lemma}

\begin{proof}
By \eqref{einv}, we have
\begin{align*}
&\bigl| |Y_{x,\gamma}(\ii)-Y_{x,\gamma}(\jj)|-|Y_{x,\gamma}(\left.\ii\right|_N\mathbf{0})-Y_{x,\gamma}(\left.\jj\right|_N\mathbf{0})|\bigr|\\
&\leq\left|\left(Y_{x,\gamma}(\ii)-Y_{x,\gamma}(\left.\ii\right|_N\mathbf{0})\right)-\left(Y_{x,\gamma}(\jj)-Y_{x,\gamma}(\left.\jj\right|_N\mathbf{0})\right)\right|\\
&\leq\gamma^{N}\left|Y_{v_{\left.\ii\right|_N}(x),\gamma}(\sigma^{N}(\ii))-Y_{v_{\left.\ii\right|_N}(x),\gamma}(\mathbf{0})\right|+\gamma^{N}\left|Y_{v_{\left.\jj\right|_N}(x),\gamma}(\sigma^{N}(\jj))-Y_{v_{\left.\jj\right|_N}(x),\gamma}(\mathbf{0})\right|\\
&\leq\gamma^N\frac{8\pi\gamma}{1-\gamma}< {\tilde\gamma_b}^N\frac{8\pi\tilde\gamma_b}{1-\tilde\gamma_b} \leq r
\end{align*}
for sufficiently large $N = N(r)$, which implies the inequality \eqref{eineq}.
\end{proof}

To use Corollary~\ref{cor_led}, we show the following.
\begin{prop}\label{pac}
For Lebesgue almost every $(x, \gamma) \in(0,1) \times (1/b, \tilde\gamma_b)$ the measure $m_{x,\gamma}$ is absolutely continuous $($in particular, $\dim m_{x,\gamma} = 1)$.
\end{prop}

\begin{proof} Take $\varepsilon > 0$.
We will prove that $m_{x,\gamma}$ is absolutely continuous for Lebesgue almost every $(x,\gamma)\in R_{\varepsilon}$, where
\[
R_\varepsilon =[0,1)\times\left(1/b + \varepsilon,\tilde\gamma_b - \varepsilon\right)
\]
for an arbitrarily small $\varepsilon>0$, which will imply the statement. Denote by
\[
\underline{D}(m_{x,\gamma},y)=\liminf_{r\rightarrow0}\frac{m_{x,\gamma}((y-r, y+r))}{2r}
\]
the lower density of the measure $m_{x,\gamma}$ at a point $y \in \mathbb R$. By \cite[Theorem 2.12]{Mat}, if $\underline{D}(m_{x,\gamma},y)<\infty$ for $m_{x,\gamma}$-almost every $y$, then the measure $m_{x,\gamma}$ is absolutely continuous. Therefore, it is enough to show that
\begin{equation*}
\mathcal{I}:=\iint_{R_{\varepsilon}}\int_{\R}\underline{D}(m_{x,\gamma},y)\: dm_{x,\gamma}(y)d\mathcal{L}_2(x,\gamma)<\infty.
\end{equation*}
By the Fatou Lemma, the Fubini Theorem and the definition of the measure $m_{x,\gamma}$, we have
\begin{equation}\label{I}
\mathcal{I}\leq\liminf_{r\rightarrow0}\frac{1}{2r}\iint_{\Sigma\times\Sigma}\mathcal{L}_2\left(\left\{(x,\gamma)\in R_{\varepsilon}:\left|Y_{x,\gamma}(\ii)-Y_{x,\gamma}(\jj)\right|<r\right\}\right)d\mathbb{P}(\ii)d\mathbb{P}(\jj).
\end{equation}
We can write
\begin{multline}\label{II}
\iint_{\Sigma\times\Sigma}\mathcal{L}_2\left(\left\{(x,\gamma)\in R_{\varepsilon}:\left|Y_{x,\gamma}(\ii)-Y_{x,\gamma}(\jj)\right|<r\right\}\right)d\mathbb{P}(\ii)d\mathbb{P}(\jj)\\
=\sum_{n=0}^{\infty}\sum_{\overline{k}\in\left\{0,\ldots, b-1\right\}^n}\iint_{\A_{\overline{k}}}\mathcal{L}_2\left(\left\{(x,\gamma)\in R_{\varepsilon}:\left|Y_{x,\gamma}(\ii)-Y_{x,\gamma}(\jj)\right|<r\right\}\right)d\mathbb{P}(\ii)d\mathbb{P}(\jj).
\end{multline}

Let
\[
R_{\overline{k},\varepsilon}=[v_{\overline{k}}(0),v_{\overline{k}}(1))\times(1/b+\varepsilon,\tilde\gamma_b-\varepsilon).
\]
By \eqref{einv} and Lemma~\ref{ldecomp}, for any $\ii,\jj\in \A_{\overline{k}}$ we have
\begin{align*}
&\mathcal{L}_2\left(\left\{(x,\gamma)\in R_{\varepsilon}:\left|Y_{x,\gamma}(\ii)-Y_{x,\gamma}(\jj)\right|<r\right\}\right)\\
&=\mathcal{L}_2\left(\left\{(x,\gamma)\in R_{\varepsilon}:\left|Y_{v_{\overline{k}}(x),\gamma}(\sigma^n(\ii))-Y_{v_{\overline{k}}(x),\gamma}(\sigma^n(\jj))\right|<r\gamma^{-n}\right\}\right)\\
&=b^n\mathcal{L}_2\left(\left\{(v,\gamma)\in R_{\overline{k},\varepsilon}:\left|Y_{v,\gamma}(\sigma^n(\ii))-Y_{v,\gamma}(\sigma^n(\jj))\right|<r\gamma^{-n}\right\}\right)\\
&\leq b^n\mathcal{L}_2\left(\left\{(v,\gamma)\in R_{\overline{k},\varepsilon}:\left|Y_{v,\gamma}(\sigma^n(\ii))-Y_{v,\gamma}(\sigma^n(\jj))\right|<r\left( \frac 1 b +\varepsilon\right)^{-n}\right\}\right)\\
&\leq b^n\mathcal{L}_2\left(\left\{(v,\gamma)\in R_{\overline{k},\varepsilon}:\left|Y_{v,\gamma}(\left.\sigma^n(\ii)\right|_N\mathbf{0})-Y_{v,\gamma}(\left.\sigma^n(\jj)\right|_N\mathbf{0})\right|<2r\left( \frac 1 b +\varepsilon\right)^{-n}\right\}\right),
\end{align*}
where $N$ depends on $n,r$. Hence,
\begin{align*}
&\sum_{\overline{k}\in\left\{0,\ldots, b-1\right\}^n}\iint_{\A_{\overline{k}}}\mathcal{L}_2\left(\left\{(x,\gamma)\in R_{\varepsilon}:\left|Y_{x,\gamma}(\ii)-Y_{x,\gamma}(\jj)\right|<r\right\}\right)d\mathbb{P}(\ii)d\mathbb{P}(\jj)\\
&\le b^n\sum_{\overline{k}\in\left\{0,\ldots, b-1\right\}^n}\\
&\iint_{\A_{\overline{k}}}
\mathcal{L}_2\left(\left\{(v,\gamma)\in R_{\overline{k},\varepsilon}:\left|Y_{v,\gamma}(\left.\sigma^n(\ii)\right|_N\mathbf{0})-Y_{v,\gamma}(\left.\sigma^n(\jj)\right|_N\mathbf{0})\right|<2r\left( \frac 1 b +\varepsilon\right)^{-n}\right\}\right)d\mathbb{P}(\ii)d\mathbb{P}(\jj)\\
&=\frac{1}{b^{n+2N}}
\sum_{\overline{k}\in\left\{0,\ldots,b-1\right\}^n}\sum_{(\overline{l},\overline{m})\in \left.\A_{\emptyset}\right|_N} \mathcal{L}_2\left(\left\{(v,\gamma)\in R_{\overline{k},\varepsilon}:\left|Y_{v,\gamma}(\overline{l}\mathbf{0})-Y_{v,\gamma}(\overline{m}\mathbf{0})\right|<2r \left( \frac 1 b +\varepsilon\right)^{-n}\right\}\right)\\
&=\frac{1}{b^{n+2N}}\sum_{(\overline{l},\overline{m})\in\left.\A_{\emptyset}\right|_N}\mathcal{L}_2\left(\left\{(x,\gamma)\in R_{\varepsilon}:\left|Y_{x,\gamma}(\overline{l}\mathbf{0})-Y_{x,\gamma}(\overline{m}\mathbf{0})\right|<2r \left( \frac 1 b +\varepsilon\right)^{-n}\right\}\right),
\end{align*}
where in the latter equality we used
\[
R_{\varepsilon}=\bigsqcup_{\overline{k}\in\left\{0,\ldots, b-1\right\}^n}R_{\overline{k},\varepsilon}.
\]
Hence, using \eqref{I}, \eqref{II} and Lemma~\ref{ltransleb}, we get
\begin{align*}
\mathcal{I}&\leq
\liminf_{r\rightarrow0}\frac{1}{2r}\sum_{n=0}^{\infty}
\frac{1}{b^{n+2N}}\sum_{(\overline{l},\overline{m})\in\left.\A_{\emptyset}\right|_N}\mathcal{L}_2\left(\left\{(x,\gamma)\in R_{\varepsilon}:\left|Y_{x,\gamma}(\overline{l}\mathbf{0})-Y_{x,\gamma}(\overline{m}\mathbf{0})\right|<2r \left( \frac 1 b +\varepsilon\right)^{-n}\right\}\right)\\&\leq
\liminf_{r\rightarrow0}\frac{1}{2r}\sum_{n=0}^{\infty}\frac{1}{b^{n+2N}} \sum_{(\overline{l},\overline{m})\in\left.\A_{\emptyset}\right|_N} 2Cr \left( \frac 1 b +\varepsilon\right)^{-n} = C\sum_{n=0}^{\infty}(1 + b \varepsilon)^{-n},
\end{align*}
which is finite since $\varepsilon>0$.
\end{proof}

Now Theorem~C follows directly from Proposition~\ref{pac}, the Fubini Theorem, \eqref{gamma}, \eqref{tildegamma} and Corollary~\ref{cor_led}.

To obtain more precise estimates of $\tilde\lambda_2$, $\tilde\lambda_3$, $\tilde\lambda_4$ presented in Corollary~D, one needs to find suitable $(*)$-functions. To this aim, we use the following result by Peres and Solomyak.

\begin{lemma}[Peres, Solomyak \mbox{\cite[Lemma~5.1]{peso2}}]\label{lpeso}
Let $\beta\geq1$. Suppose that for some positive integer $k=k(\beta)$ and a real number $\eta=\eta(\beta)$ there exists a function $g_{\beta}: \mathbb R \to \mathbb R$,
\[
g_{\beta}(t)=1-\beta\sum_{n=1}^{k-1}t^n+\eta t^k+\beta\sum_{n=k+1}^{\infty}t^n
\]
such that for some $t_\beta\in(0,1)$,
\[
g_\beta(t_\beta)>0\quad \text{and} \quad g_\beta'(t_\beta)<0.
\]
Then $y(\beta)>t_\beta$. More precisely, there exists $\varepsilon>0$ such that for every $g\in\mathcal{G}_\beta$ and every $t\in(0,t_\beta)$,
\[
g(t)<\varepsilon \quad \Rightarrow \quad g'(t)<-\varepsilon.
\]
\end{lemma}

Let
\[
\beta = \frac{1}{\sqrt{\sin^2(\pi/b) - 1/(b^2\lambda - 1)^2}}.
\]
and consider functions $g_\beta$ defined in Lemma~\ref{lpeso}.

For $b = 2$ take $k = 4$, $\eta = 0.81$, $\lambda = 0.81$. Then $g_\beta(0.62) > 0$ and $g'_\beta(0.62) < 0$, so $y(\beta) > 0.62$. On the other hand, $1/(2\lambda) = 1/1.62 < 0.62$. By Lemma~\ref{tilde}, $\tilde\lambda_2 < 0.81$.

For $b = 3$ take $k = 4$, $\eta = 1.43398$, $\lambda = 0.55$. Then $g_{\beta}(0.6061) > 0$ and $g'_{\beta}(0.6061) < 0$, so $y(\beta) > 0.6061$. On the other hand, $1/(3\lambda) = 1/1.65 < 0.6061$. By Lemma~\ref{tilde}, $\tilde\lambda_3 < 0.55$.

For $b = 4$ take $k = 3$, $\eta = -0.298$, $\lambda = 0.44$. Then $g_{\beta}(0.569) > 0$ and $g'_{\beta}(0.569) < 0$, so $y(\beta) > 0.569$. On the other hand, $1/(4\lambda) = 1/1.76 < 0.569$. By Lemma~\ref{tilde}, $\tilde\lambda_4 < 0.44$.

\end{document}